\documentclass[11pt]{article}
\usepackage{amsmath}
\usepackage{amsthm}
\usepackage{amssymb}
\usepackage{amscd}
\usepackage{relsize}
\usepackage[pdftex]{graphicx}
\usepackage{enumerate}
\newtheorem{teo}{Teorema}

\newtheorem{cor}[teo]{Corollary}
\newtheorem{lema}[teo]{Lemma}
\newcommand{\CC}{\mathbb{C}}
\newcommand{\RR}{\mathbb{R}}
\newcommand{\ZZ}{\mathbb{Z}}

\newcommand{\Hc}{\mathcal{H}}
\newcommand{\Tc}{\mathcal{T}}

\newcommand{\Ac}{\mathcal{A}}

\newcommand{\Cc}{\mathcal{C}}
\newcommand{\Lc}{\mathcal{L}}

\DeclareMathOperator*{\esup}{ess\,sup}
\DeclareMathOperator*{\einf}{ess\,inf}
\title{{\bf Semi-direct product of groups, filter banks and sampling}}
\author{
{\bf A.~G. Garc\'{\i}a}\thanks{E-mail:\texttt{agarcia@math.uc3m.es}}, \,\,
{\bf M.~A. Hern\'andez-Medina}\thanks{E-mail:\texttt{miguelangel.hernandez.medina@upm.es}}
{\bf\,\, and \,\, G. P\'erez-Villal\'on}\thanks{E-mail:\texttt{gperez@euitt.upm.es}}
}
\date{}
\begin{document}
\maketitle
\begin{itemize}
\item[*] Departamento de Matem\'aticas, Universidad Carlos III de Madrid,
 Avda. de la Universidad 30, 28911 Legan\'es-Madrid, Spain.
\item[\dag] Information Processing and Telecommunications Center, Universidad Polit\'ecnica de Madrid, Departamento de Matem\'atica Aplicada a las Tecnolog\'{\i}as de la Informaci\'on y las Comunicaciones, E.T.S.I.T., Avda. Complutense 30, 28040 Madrid, Spain.
 \item[\ddag] Departamento de Matem\'atica Aplicada a las Tecnolog\'{\i}as de la Informaci\'on y las Comunicaciones, E.T.S.I.T., Universidad Polit\'ecnica de Madrid,
 Avda. Complutense 30, 28040 Madrid, Spain.
\end{itemize}
%%%%%%%%%%%%%%%%%%%%
\begin{abstract}
An abstract sampling theory associated to a unitary representation of a countable discrete non abelian group $G$, which is a semi-direct product of groups, on a separable Hilbert space is studied. A suitable expression of the data samples and the use of a filter bank formalism allows to fix the mathematical problem to be solved: the search of appropriate dual frames for $\ell^2(G)$. An example involving  crystallographic groups illustrates the obtained results by using average or pointwise samples. 
\end{abstract}
%%%%%%%%%%%%%%%%%%%%%%%%%%%%%%%%%%%%%%%%%%%%%%%%%%%%%%%%%%%%%%%%%%%%%%

\noindent{\bf AMS}: 42C15; 94A20; 22B05; 20H15.
%%%%%%%%%%%%%%
\section{Statement of the problem}
\label{section1}
%%%%%%%%%%%%%%
In this paper an abstract sampling theory associated to non abelian groups is derived for the specific case of a unitary representation of a semi-direct product of groups on a separable Hilbert space. Semi-direct product of groups provide important examples of non abelian groups such as  dihedral groups, infinite dihedral group, euclidean motion groups or crystallographic groups. Concretely, let $(n,h)\mapsto U(n,h)$ be a unitary representation on a separable Hilbert space $\Hc$ of a  semi-direct product $G=N\rtimes_\phi H$, where $N$ is a countable discrete LCA (locally compact abelian) group, $H$ is a finite group, and $\phi$ denotes the action of the group $H$ on the group $N$ (see Section \ref{section2} infra for the details); for a fixed $a\in \Hc$ we consider the $U$-invariant subspace in $\Hc$
\[
\Ac_a=\Big\{ \sum_{(n,h)\in G} \alpha(n,h)\,U(n,h)a\,\, :\,\, \{\alpha(n,h)\}_{(n,h)\in G}\in \ell^2(G) \Big\}\,,
\]
where we assume that $\{U(n,h)a\}$ is a Riesz sequence for $\Hc$, i.e., a Riesz basis for $\Ac_a$ (see Ref.~\cite{barbieri:15} for a necessary and sufficient condition). Given $K$ elements $b_k$ in $\Hc$, which do not belong necessarily to 
$\Ac_a$, the main goal in this paper is the stable recovery of any $x\in \Ac$ from the given data (generalized samples)
\[
\Lc_kx(n):=\big\langle x, U(n,1_H)b_k \big\rangle_\Hc\,, \quad \text{ $n\in N$ and $k=1,2,\dots, K$}\,,
\]
where $1_H$ denotes the identity element in $H$. These samples are nothing but a generalization of average sampling in shift-invariant subspaces of $L^2(\RR^d)$; see, among others, Refs.~\cite{aldroubi:05,hector:14,garcia:06,garcia:09,kang:11,michaeli:11,pohl:12,sun:03}.
The case where $G$ is a discrete LCA group and the samples are taken at a uniform lattice of $G$ has been solved in Ref.~\cite{garcia:17}; this work relies on the use of the Fourier analysis in the LCA group $G$. In the case involved here  a Fourier analysis is not available and, consequently, we need to overcome this drawback. 

Having in mind  the  filter bank formalism in discrete LCA  groups (see, for instance, Refs.~\cite{bol:98,vetterli:98,garcia:18}), the given data 
$\{\Lc_kx(n)\}_{n\in N;\,k=1,2,\dots K}$ can be expressed as the output of a suitable $K$-channel analysis filter bank corresponding to the input $\boldsymbol{\alpha}=\{\alpha(n,h)\}_{(n,h)\in G}$ in $\ell^2(G)$. As a consequence, the problem consists of finding a synthesis part of  the former filter bank allowing perfect reconstruction; besides only Fourier analysis on the LCA group $N$ is needed. Then, roughly speaking, substituting the output of the synthesis part in $x=\sum_{(n,h)\in G} \alpha(n,h)\,U(n,h)a$ we will obtain the corresponding sampling formula in $\Ac_a$. 

\medskip

This said, as it could be expected the problem can be mathematically formulated as the search of dual frames for $\ell^2(G)$ having the form 
\[
\big\{T_n \mathsf{h}_k\big\}_{n\in N;\,k=1,2,\dots K}\quad \text{and}\quad \big\{T_n\mathsf{g}_k\big\}_{n\in N;\,k=1,2,\dots K}\,. 
\]
Here $\mathsf{h}_k, \mathsf{g}_k \in \ell^2(G)$, $T_n \mathsf{h}_k(m,h)=\mathsf{h}_k(m-n,h)$ and $T_n \mathsf{g}_k(m,h)=\mathsf{g}_k(m-n,h)$, $(m,h)\in G$, where $n\in N$ and $k=1,2,\dots,K$. Besides, for any $x\in \Ac_a$ we have the expression for its samples 
\[
\Lc_k x(n)=\big\langle \boldsymbol{\alpha}, T_n \mathsf{h}_k \big\rangle_{\ell^2(G)}\,,\quad  \text{$n\in N$ and $k=1,2,\dots,K$}\,.
\] 
Needless to say that frame theory plays a central role in what follows; the necessary background on Riesz bases or frame theory in a separable Hilbert space can be found, for instance, in Ref.~\cite{ole:16}. Finally, sampling formulas in $\Ac_a$ having the form
\[
x=\sum_{k=1}^K\sum_{n\in N} \Lc_k x(n)\,U(n,1_H)c_k \quad \text{ in $\Hc$}\,,
\]
for some $c_k\in \Ac_a$, $k=1,2,\dots,K$, will come out by using, for $\mathsf{g}\in \ell^2(G)$ and $n\in N$, the shifting property $\Tc_{U,a}\big(T_n \mathsf{g}\big)=U(n, 1_H) \big(\Tc_{U,a}\mathsf{g}\big)$ that satisfies the natural isomorphism $\Tc_{U,a}\,:\,\ell^2(G) \rightarrow \Ac_a$ which maps the usual orthonormal basis 
$\{\boldsymbol{\delta}_{(n,h)}\}_{(n,h)\in G}$ for $\ell^2(G)$ onto the Riesz basis $\big\{U(n,h)a\big\}_{(n,h)\in G}$ for $\Ac_a$. All these steps will be carried out throughout the remaining sections. For the sake of completeness, Section \ref{section2} includes some basic preliminaries on semi-direct product of groups and Fourier analysis on LCA groups.
The paper ends with an illustrative example involving the quasi regular representation  of a crystallographic group on 
$L^2(\RR^d)$; sampling formulas involving average or pointwise samples are obtained for the corresponding $U$-invariant subspaces in $L^2(\RR^d)$.
%%%%%%%%%%%%%%%%%%
\section{Some mathematical preliminaries}
\label{section2}
%%%%%%%%%%%%%%%%%%
In this section we introduce the basic tools in semi-direct product of groups and in harmonic analysis in a discrete LCA group that they  will be used in the sequel.
%%%%%%%%%%%%%%%%%%%%%%%%%%%%%%%%%
\subsection{Preliminaries on semi-direct product of groups}
%%%%%%%%%%%%%%%%%%%%%%%%%%%%%%%%%
Given groups $(N, \cdot)$ and $(H,\cdot)$, and a homomorphism $\phi: H \rightarrow Aut(N)$ their semi-direct product 
$G:=N\rtimes_\phi H$ is defined as follows: The underlying set of $G$ is the set of pairs $(n,h)$ with $n\in N$ and $h\in H$, along with the multiplication rule 
\[
(n_1,h_1)\cdot(n_2,h_2):=(n_1\phi_{h_1}(n_2),h_1h_2)\,, \quad (n_1,h_1), \,(n_2,h_2) \in G\,, 
\]
where we denote $\phi(h):=\phi_h$; usually the homomorphism $\phi$ is referred as the action of the group $H$ on the group $N$. Thus we obtain a new group with identity element $(1_N,1_H)$, and inverse 
$(n,h)^{-1}=(\phi_{h^{-1}}(n^{-1}),h^{-1})$. 

Besides, we have the isomorphisms $N\simeq N\times \{1_H\}$ and $H\simeq \{1_N\}\times H$.  Unless $\phi_h$ equals the identity for all $h\in H$, the group $G=N \rtimes_\phi H$ is not abelian, even for abelian $N$ and $H$ groups. In case $N$ is an abelian group, it is a normal subgroup in $G$. Next we list some examples of semi-direct product of groups:
\begin{enumerate}
   \item The dihedral group $D_{2N}$ is the group of symmetries of a regular $N$-sided  polygon; it  is the semi-direct product $D_{2N}=\ZZ_N\rtimes_\phi \ZZ_2$ where 
   $\phi_{\bar{0}}\equiv Id_{\ZZ_N}$ and $\phi_{\bar{1}}(\bar{n})=-\bar{n}$ for each $\bar{n}\in \ZZ_N$. The infinite dihedral group $D_\infty$ defined as $\ZZ\rtimes_\phi\ZZ_2$ for the similar homomorphism $\phi$ is the group of isometries of 
   $\ZZ$.
   \item The Euclidean motion group $E(d)$  is the semi-direct product $\RR^d\rtimes_\phi O(d)$, where $O(d)$ is the orthogonal group of order $d$ and $\phi_A(x)=Ax$ for $A\in O(d)$ and $x\in \RR^d$. It contains as a subgroup any crystallographic group $M\ZZ^d \rtimes_\phi \Gamma$, where $M\ZZ^d$ denotes a full rank lattice of $\RR^d$ and 
   $\Gamma$ is any finite subgroup of $O(d)$ such that $\phi_\gamma(M\ZZ^d)=M\ZZ^d$ for each $\gamma \in \Gamma$.
   \item The orthogonal group $O(d)$ of all orthogonal real $d\times d$ matrices is isomorphic to the semi-direct product $SO(d) \rtimes_\phi C_2$, where $SO(d)$ consists of all orthogonal matrices with determinant $1$ and $C_2=\{I,R\}$ a cyclic group of order $2$; $\phi$ is the homomorphism given by $\phi_I(A)=A$ and  $\phi_R(A)=RAR^{-1}$ for $A\in SO(d)$.
\end{enumerate}
Suppose that $N$ is an LCA group with Haar measure $\mu_N$ and $H$ is a locally compact group with Haar measure $\mu_H$. Then, the semi-direct product 
$G=N\rtimes_\phi H$ endowed with the product topology becomes also a topological group. For the left Haar measure on $G$ see Ref.~\cite{barbieri:15}.
%%%%%%%%%%%%%%%%%%%%%%%%%%%%%%%%%%%%%%%%
\subsection{Some preliminaries on harmonic analysis on discrete LCA groups}
%%%%%%%%%%%%%%%%%%%%%%%%%%%%%%%%%%%%%%%%
The results about harmonic analysis on  locally compact abelian (LCA) groups are borrowed from Ref.~\cite{folland:95}. Notice that, in particular, a countable discrete abelian group is a second countable Hausdorff LCA group. 

\noindent For a countable discrete group $(N,\cdot)$, non necessarily abelian, the {\em convolution} of $x,y :N\rightarrow \CC$ is formally defined as $(x\ast y)(m):= \sum_{n\in N} x(n) y(n^{-1}m)$,\, $m\in N$. If in addition the group is abelian, therefore denoted by $(N,+)$, the convolution reads as
\[
(x\ast y)(m):= \sum_{n\in N} x(n)
y(m-n)\,, \quad m\in N\,.
\]
Let $\mathbb{T}=\{z\in \mathbb{C}: |z|=1\}$ be the unidimensional torus. 
We said that $\xi:N\mapsto \mathbb{T}$ is a character of $N$ if $\xi(n+m)=\xi(n)\xi(m)$  for all $n,m\in N$. We denote $\xi(n)=\langle n,\xi \rangle$.  Defining $(\xi+\gamma)(n)=\xi(n)\gamma(n)$,  the set of characters $\widehat{N}$ with the operation 
$+$ is a group, called the dual group of $N$; since $N$ is discrete 
$\widehat{N}$ is compact \cite[Prop. 4.4]{folland:95}. For $x\in \ell^1(N)$ we define its {\em Fourier transform} as
\[
X(\xi)=\widehat{x}(\xi):=\sum_{n\in N}x(n) \overline{\langle n,\xi \rangle}=\sum_{n\in N}x(n) \langle -n,\xi \rangle\,,\quad \xi\in \widehat{N}\,.
\]
It is known \cite[Theorem 4.5]{folland:95} that $\widehat{\mathbb{Z}}\cong \mathbb{T}$, with $\langle n,z \rangle = z^n$, and
$\widehat{\mathbb{Z}}_s\cong \mathbb{Z}_s:=\mathbb{Z}/s\mathbb{Z}$, with $\langle n,m \rangle = W_s^{nm}$, where $W_s=e^{2\pi i/s}$.

There exists a unique  measure, the Haar measure $\mu$ on  $\widehat{N}$ satisfying  $\mu(\xi+ E)=\mu(E)$,  for every Borel set $E\subset \widehat{N}$ \cite[Section 2.2]{folland:95},
and $\mu(\widehat{N})=1$. We denote 
$\int_{\widehat{N}} X(\xi) d\xi=\int_{\widehat{N}} X(\xi) d\mu(\xi)$.
If $N=\mathbb{Z}$, 
\[
\int_{\widehat{N}} X(\xi) d\xi=\int_{\mathbb{T}} X(z) dz= \frac{1}{2\pi}\int_0^{2\pi} X(e^{iw}) dw\,,
\]
and if $N=\mathbb{Z}_s$, 
\[
\int_{\widehat{N}} X(\xi) d\xi=\int_{\mathbb{Z}_s} X(n) dn= \frac{1}{s}\sum_{n\in \mathbb{Z}_s} X(n)\,.
\]
If $N_1, N_2,\ldots N_d$ are abelian discrete groups then the dual group of the product group is 
$\big(N_1 \times N_2 \times \ldots \times N_d\big)^{\wedge}\cong \widehat{N}_1 \times \widehat{N}_2 \times \ldots \times \widehat{N}_d$ (see \cite[Prop. 4.6]{folland:95}) with 
\[
\big\langle\, (n_1,n_2,\ldots,n_d)\, ,\, (\xi_1,\xi_2\ldots,\xi_d)\, \big\rangle = \langle n_1,\xi_1\rangle \langle n_2,\xi_2\rangle\cdots \langle n_d,\xi_d\rangle\,.
\]
The Fourier transform on $\ell^1(N)\cap \ell^2(N)$ is an isometry on a dense subspace of $L^2(\widehat{N})$; Plancherel theorem extends it in a unique manner to a unitary operator  of $\ell^2(N)$ onto $L^2(\widehat{N})$ \cite[p. 99]{folland:95}. The following lemma, giving a relationship between Fourier transform and convolution, will be used later:
\begin{lema}
 \label{conv} 
 Assume that $a, b\in \ell^2(N)$ \,and \, $\widehat{a}(\xi)\,\widehat{b}(\xi)\in L^2(\widehat{N})$. Then the convolution $a\ast b$ belongs to  $\ell^2(N)$ and
 $\widehat{a\ast b}(\xi)=\widehat{a}(\xi)\,\widehat{b}(\xi)$, \,a.e. $\xi\in \widehat{N}$.
 \end{lema}
 \begin{proof} 
 By using Plancherel theorem \cite[Theorem 4.25]{folland:95} we obtain
 \[
 \begin{split}
 (a\ast b)(n)&=\sum_{m\in N}a(m)b(n-m)=\langle a,\widetilde{b}(\cdot-n)\rangle_{\ell^2(N)}=\big\langle \widehat{a},\widehat{\widetilde{b}(\cdot-n)}\big\rangle_{L^2(\widehat{N})}\\
 &=\int_{\widehat{N}}  \widehat{a}(\xi)\,\overline{\widehat{\widetilde{b}}(\xi)} \, \overline{\langle -n,\xi \rangle}d\xi=\int_{\widehat{N}}  \widehat{a}(\xi)\,\widehat{b}(\xi) \, \overline{\langle -n,\xi \rangle}d\xi\,.
 \end{split}
 \]
 Since $\big\{\langle-n,\xi \big\rangle\}_{n\in N}$ is an orthonormal basis for $L^2(\widehat{N})$ \cite[Theorems 4.26 and 4.31]{folland:95} (we are assuming that 
 $\mu_{\widehat{N}}(\widehat{N})=1$) we finally obtain
\[
\widehat{a}(\xi)\,\widehat{b}(\xi)=\sum_{n\in N} (a\ast b)(n)\langle -n,\xi \rangle=\widehat{a\ast b}(\xi)\,,\quad \text{a.e. $\xi\in \widehat{N}$}\,.
\]
\end{proof}
%%%%%%%%%%%%%%%%%%%%%%%%%%%%%%%
\section{Filter banks formalism on semi-direct product of groups}
\label{section3}
%%%%%%%%%%%%%%%%%%%%%%%%%%%%%%%
In what follows we will assume that $G=N\rtimes_\phi H$ where $(N,+)$ is a countable discrete abelian group and $(H,\cdot)$ is a finite group. Having in mind the operational calculus
$(n,h)\cdot(m,l)=(n+\phi_h(m), hl)$,\,\, $(n,h)^{-1}=(\phi_{h^{-1}}(-n), h^{-1})$ and $(n,h)^{-1}\cdot(m,l)=(\phi_{h^{-1}}(m-n),h^{-1}l)$, the convolution $\boldsymbol{\alpha}*\mathsf{h}$ of $\boldsymbol{\alpha}, \mathsf{h}\in \ell^2(G)$ can be expressed as
\begin{equation}
\label{convolution}
\begin{split}
(\boldsymbol{\alpha}*\mathsf{h})(m,l)&=\sum_{(n,h)\in G} \alpha(n,h)\, \mathsf{h}\big[(n,h)^{-1}\cdot(m,l)\big]\\ &=\sum_{(n,h)\in G} \alpha(n,h)\, \mathsf{h}\big(\phi_{h^{-1}}(m-n),h^{-1}l\big)\,,\quad (m,l)\in G\,.
\end{split}
\end{equation}
For a function $\boldsymbol{\alpha}:G\rightarrow \CC$, its {\em $H$-decimation}  $\downarrow_H\!\boldsymbol{\alpha} :N\rightarrow \CC$ is defined as $(\downarrow_H\!\boldsymbol{\alpha})(n):=\alpha(n,1_H)$ for any $n\in N$. Thus we have
\begin{equation}
\label{dec}
\begin{split}
\downarrow_H\!(\boldsymbol{\alpha}*\mathsf{h})(m)&=(\boldsymbol{\alpha}*\mathsf{h})(m,1_H)=\sum_{(n,h)\in G} \alpha(n,h)\, \mathsf{h}\big(\phi_{h^{-1}}(m-n),h^{-1}\big)\\
&=\sum_{(n,h)\in G} \alpha(n,h)\, \mathsf{h}[(n-m,h)^{-1}]\,,\quad m\in N\,.
\end{split}
\end{equation}
Defining the polyphase components of $\boldsymbol{\alpha}$ and $\mathsf{h}$ as $\boldsymbol{\alpha}_h(n):=\alpha(n,h)$ and $\mathsf{h}_h(n):=\mathsf{h}[(-n,h)^{-1}]$ respectively, we write
\[
\downarrow_H\!(\boldsymbol{\alpha}*\mathsf{h})(m)=\sum_{h\in H}\sum_{n\in N} \boldsymbol{\alpha}_h(n)\,\mathsf{h}_h(m-n)=\sum_{h\in H} \big(\boldsymbol{\alpha}_h*_N \mathsf{h}_h\big)(m)\,, \quad m\in N\,.
\]
For a function $c: N\rightarrow \CC$, its {\em $H$-expander} $\uparrow_H\!c :G\rightarrow \CC$ is defined as
\[
(\uparrow_H\! c)(n,h)= \begin{cases}c(n)& \text{if $h=1_H$}\\
0& \text{if $h\neq1_H$}\,.\end{cases}
\]
In case $\uparrow_H\! c$ and $\mathsf{g}$ belong to $\ell^2(G)$ we have
\[
\begin{split}
(\uparrow_H\!c*\mathsf{g})(m,l)&=\sum_{(n,h)\in G} (\uparrow_H\!c)(n,h)\, \mathsf{g}\big[(n,h)^{-1}\cdot(m,l)\big]\\
&=\sum_{(n,h)\in G} (\uparrow_H\!c)(n,h)\, \mathsf{g}\big(\phi_{h^{-1}}(m-n),h^{-1}l\big)\\
&=\sum_{n\in N} c(n)\,\mathsf{g}(m-n,l)=\big(c*_N\mathsf{g}_l\big)(m)\,,\quad m\in N\,,\,l\in H\,,
\end{split}
\]
where $\mathsf{g}_l(n):=\mathsf{g}(n,l)$ is the polyphase component of $\mathsf{g}$.

From now on  we will refer to a {\em $K$-channel filter bank} with {\em analysis filters} $\mathsf{h}_k$ and {\em synthesis filters} $\mathsf{g}_k$, $k=1,2,\dots,K$ as the one given by (see Fig.~\ref{f1})
\begin{equation}
\label{fbank}
\mathbf{c}_k:=\downarrow_H\!(\boldsymbol{\alpha}*\mathsf{h}_k)\,, \,\,\text{$k=1,2,\dots,K$\,,\,\, and}\quad \boldsymbol{\beta}=\sum_{k=1}^K(\uparrow_H\!c_k)*\mathsf{g}_k\,,
\end{equation}
where $\boldsymbol{\alpha}$ and $\boldsymbol{\beta}$ denote, respectively, the input and the output of the filter bank. In polyphase notation,
\begin{equation}
\label{fbankpol}
\begin{split}
\mathbf{c}_k(m)&=\sum_{h\in H} \big(\boldsymbol{\alpha}_h*_N \mathsf{h}_{k,h}\big)(m)\,, \quad m\in N\,, \quad k=1,2,\dots,K\,,\\
\boldsymbol{\beta}_l(m)&=\sum_{k=1}^K \big(\mathbf{c}_k*_N\mathsf{g}_{l,k}\big)(m)\,,\quad m\in N\,,\,\, l\in H\,,
\end{split}
\end{equation}
where $\boldsymbol{\alpha}_h(n):=\alpha(n,h)$, $\boldsymbol{\beta}_l(n):=\beta(n,l)$, $\mathsf{h}_{k,h}(n):=\mathsf{h}_k[(-n,h)^{-1}]$ and $\mathsf{g}_{l,k}(n):=\mathsf{g}_k(n,l)$ are the {\em polyphase components} of $\boldsymbol{\alpha}$, $\boldsymbol{\beta}$, $\mathsf{h}_k$ and $\mathsf{g}_k$, $k=1,2,\dots,K$, respectively.
We also assume that $\mathsf{h}_k, \mathsf{g}_k \in \ell^2(G)$ with $\widehat{\mathsf{h}}_{k,h}, \widehat{\mathsf{g}}_{h,k} \in L^\infty(\widehat{N})$ for $k=1,2,\dots,K$ and $h\in H$; from Lemma \ref{conv} the filter bank \eqref{fbank} is well defined in $\ell^2(G)$.

\begin{figure}
\begin{center}
\includegraphics[viewport=5cm 23.4cm 12cm 25.6cm]{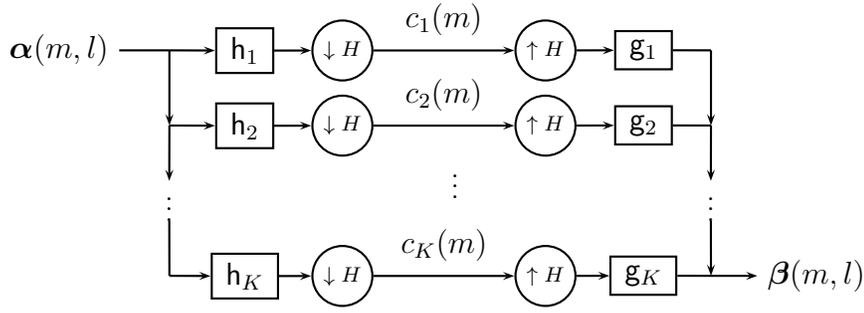}
\vspace{1cm}
 \caption{The K-channel filter bank scheme}\label{f1}
\end{center}
\end{figure}

\medskip 

The above $K$-channel filter bank \eqref{fbank} is said to be a {\em perfect reconstruction} filter bank if and only if it satisfies 
$\boldsymbol{\alpha}=\sum_{k=1}^K(\uparrow_H\!c_k)*\mathsf{g}_k$ for each $\boldsymbol{\alpha}\in \ell^2(G)$, or equivalently, $\boldsymbol{\alpha}_h=\sum_{k=1}^K \big(\mathbf{c}_k*_N\mathsf{g}_{h,k}\big)$ for each $h\in H$.

\medskip

Since $N$ is an LCA group where a Fourier  transform is available, the polyphase expression \eqref{fbankpol} of the filter bank \eqref{fbank} allows us to carry out its polyphase analysis.
%%%%%%%%%%%%%%%%%%%%%%%%
\subsection{Polyphase analysis: Perfect reconstruction condition}
%%%%%%%%%%%%%%%%%%%%%%%%
For notational ease, we denote $L:=|H|$, the order of the group $H$, and its elements as $H=\{h_1, h_2, \dots, h_L\}$. Having in mind Lemma \ref{conv},
the $N$-Fourier transform in $\mathbf{c}_k(m)=\sum_{h\in H} \big(\boldsymbol{\alpha}_h*_N \mathsf{h}_{k,h}\big)(m)$ gives $\widehat{\mathbf{c}}_k(\gamma)=\sum _{h\in H} \widehat{\mathsf{h}}_{k,h}(\gamma)\,\widehat{\boldsymbol{\alpha}}_h(\gamma)$ a.e. $\gamma \in \widehat{N}$ for each $k=1,2,\dots,K$. In matrix notation,
\[
\mathbf{C}(\gamma)= \mathbf{H}(\gamma)\, \mathbf{A}(\gamma)\quad \text{a.e.\,\, $\gamma \in \widehat{N}$}\,,
\]
where $\mathbf{C}(\gamma)=\big(\widehat{\mathbf{c}}_1(\gamma),\widehat{\mathbf{c}}_2(\gamma), \dots,\widehat{\mathbf{c}}_K(\gamma)\big)^\top$\,,\,\,$\mathbf{A}(\gamma)=\big(
\widehat{\boldsymbol{\alpha}}_{h_1}(\gamma), \widehat{\boldsymbol{\alpha}}_{h_2}(\gamma),\dots , \widehat{\boldsymbol{\alpha}}_{h_L}(\gamma)\big)^\top$, and  
$\mathbf{H}(\gamma)$ is the $K\times L$ matrix
\begin{equation}
\label{Hpol}
\mathbf{H}(\gamma)=\begin{pmatrix}
\widehat{\mathsf{h}}_{1,h_1}(\gamma) &  \widehat{\mathsf{h}}_{1,h_2}(\gamma) & \cdots & \widehat{\mathsf{h}}_{1,h_L}(\gamma) \\
\widehat{\mathsf{h}}_{2,h_1}(\gamma) &  \widehat{\mathsf{h}}_{2,h_2}(\gamma) & \cdots & \widehat{\mathsf{h}}_{2,h_L}(\gamma) \\
\cdots &  \cdots & \cdots & \cdots \\
\widehat{\mathsf{h}}_{K,h_1}(\gamma) &  \widehat{\mathsf{h}}_{K,h_2}(\gamma) & \cdots & \widehat{\mathsf{h}}_{K,h_L}(\gamma) \\
\end{pmatrix},
\end{equation}
where  $\widehat{\mathsf{h}}_{k,h_i}\in L^2(\widehat{N})$ is the Fourier transform of $\mathsf{h}_{k,h_i}(n):=\mathsf{h}_k[(-n,h_i)^{-1}]\in \ell^2(N)$. 

\medskip

\noindent The same procedure for $\boldsymbol{\beta}_l(m)=\sum_{k=1}^K \big(\mathbf{c}_k*_N\mathsf{g}_{l,k}\big)(m)$ gives $\widehat{\boldsymbol{\beta}}_l(\gamma)=\sum _{k=1}^K \widehat{\mathsf{g}}_{l,k}(\gamma)\,\widehat{\mathbf{c}}_k(\gamma)$ a.e. $\gamma \in \widehat{N}$. In matrix notation,
\[
\mathbf{B}(\gamma)= \mathbf{G}(\gamma)\, \mathbf{C}(\gamma)\quad \text{a.e.\,\, $\gamma \in \widehat{N}$}\,,
\]
where $\mathbf{B}(\gamma)=\big(\widehat{\boldsymbol{\beta}}_{h_1}(\gamma), \widehat{\boldsymbol{\beta}}_{h_2}(\gamma),\dots , \widehat{\boldsymbol{\beta}}_{h_L}(\gamma)\big)^\top$\,,\,\, $\mathbf{C}(\gamma)=\big(\widehat{\mathbf{c}}_1(\gamma),\widehat{\mathbf{c}}_2(\gamma), \dots,\widehat{\mathbf{c}}_K(\gamma)\big)^\top$ and $\mathbf{G}(\gamma)$ is the $L\times K$ matrix
\begin{equation}
\label{Gpol}
\mathbf{G}(\gamma)=\begin{pmatrix}
\widehat{\mathsf{g}}_{h_1,1}(\gamma) &  \widehat{\mathsf{g}}_{h_1,2}(\gamma) & \cdots & \widehat{\mathsf{g}}_{h_1,K}(\gamma) \\
\widehat{\mathsf{g}}_{h_2,1}(\gamma) &  \widehat{\mathsf{g}}_{h_2,2}(\gamma) & \cdots & \widehat{\mathsf{g}}_{h_2,K}(\gamma) \\
\cdots &  \cdots & \cdots & \cdots \\
\widehat{\mathsf{g}}_{h_L,1}(\gamma) &  \widehat{\mathsf{g}}_{h_L,2}(\gamma) & \cdots & \widehat{\mathsf{g}}_{h_L,K}(\gamma) \\
\end{pmatrix},
\end{equation}
where  $\widehat{\mathsf{g}}_{h_i,k}\in L^2(\widehat{N})$ is the Fourier transform of $\mathsf{g}_{h_i,k}(n):=\mathsf{g}_k(n,h_i)\in\ell^2(N)$. 

\medskip

Thus, in terms of the {\em polyphase matrices} $\mathbf{G}(\gamma)$ and $\mathbf{H}(\gamma)$ the filter bank \eqref{fbank} can be expressed as
\begin{equation}
\label{pfb}
\mathbf{B}(\gamma)= \mathbf{G}(\gamma)\,\mathbf{H}(\gamma)\, \mathbf{A}(\gamma)\quad \text{a.e.\,\, $\gamma \in \widehat{N}$}\,.
\end{equation}
As a consequence of \eqref{pfb} we have:
\begin{teo}
\label{PR}
The $K$-channel filter bank given in \eqref{fbank}, where $\mathsf{h}_k, \mathsf{g}_k$ belong to $ \ell^2(G)$ and $\widehat{\mathsf{h}}_{k,h_i}, \widehat{\mathsf{g}}_{h_i,k}$ belong to $ L^\infty(\widehat{N})$ for $k=1,2,\dots,K$ and $i=1,2,\dots,L$, satisfies the perfect reconstruction property if and only if  $\mathbf{G}(\gamma) \,\mathbf{H}(\gamma) =\mathbf{I}_L$\, a.e. $\gamma \in \widehat{N}$, where $\mathbf{I}_L$ denotes the identity matrix of order $L$.
\end{teo}
\begin{proof}
First of all, note that the mapping $\boldsymbol{\alpha}\in \ell^2(G) \mapsto \mathbf{A}\in L^2_L(\widehat{N})$ is a unitary operator. Indeed, for each $\boldsymbol{\alpha},\boldsymbol{\beta}\in \ell^2(G)$ we have the isometry property
\[
\begin{split}
\langle \boldsymbol{\alpha},\boldsymbol{\beta}\rangle_{\ell^2(G)}&=\sum_{(m,h)\in G} \alpha(m,h)\,\overline{\beta(m,h)}=\sum_{h\in H}\langle \boldsymbol{\alpha}_h,\boldsymbol{\beta}_h\rangle_{\ell^2(N)}\\
&=\sum_{h\in H}\langle \widehat{\boldsymbol{\alpha}}_h,\widehat{\boldsymbol{\beta}}_h\rangle_{L^2(\widehat{N})}=\langle \mathbf{A}, \mathbf{B}\rangle_{L^2_L(\widehat{N})}\,.
\end{split}
\]
It is also surjective since the $N$-Fourier transform is a surjective isometry between $\ell^2(N)$ and $L^2(\widehat{N})$.
Having in mind this property,  Eq.~\eqref{pfb} tell us that the filter bank satisfies the perfect reconstruction property if and only if $\mathbf{G}(\gamma)\,\mathbf{H}(\gamma)=\mathbf{I}_L$ a.e. $\gamma \in \widehat{N}$. 
\end{proof}

Notice that, in the perfect reconstruction setting, the number of channels $K$ must be necessarily bigger or equal that the order $L$ of the group $H$, i.e., $K\geq L$.

%%%%%%%%%%%%%%%%%%%%%%%%
\section{Frame analysis}
%%%%%%%%%%%%%%%%%%%%%%%%
For $m\in N$ the {\em translation operator} $T_m: \ell^2(G) \rightarrow \ell^2(G)$ is defined as
\begin{equation}
\label{translation}
T_m \boldsymbol{\alpha}(n,h):=\alpha\big((m,1_H)^{-1}\cdot(n,h)\big)=\alpha(n-m,h)\,, \,(n,h)\in G\,.
\end{equation}
The {\em involution operator} $\boldsymbol{\alpha}\in \ell^2(G) \mapsto \widetilde{\boldsymbol{\alpha}}\in \ell^2(G)$ is defined as 
$\widetilde{\alpha}(n,h):=\overline{\alpha\big((n,h)^{-1} \big)}$, $(n,h)\in G$. As expected, the classical relationship between convolution and translation operators holds. Thus, for the $K$-channel filter bank \eqref{fbank} we have (see \eqref{dec}):
\[
\mathbf{c}_k(m)=\downarrow_H\!(\boldsymbol{\alpha}*\mathsf{h}_k)(m)=\big\langle \boldsymbol{\alpha}, T_m \widetilde{\mathsf{h}}_k \big\rangle_{\ell^2(G)}\,, \quad m\in N\,, \,\, k=1,2,\dots,K\,.
\]
Besides,
\[
(\uparrow_H\!\mathbf{c}_k*\mathsf{g}_k)(m,h)=\sum_{n\in N} \mathbf{c}_k(n)\,\mathsf{g}_k(m-n,h)=\sum_{n\in N}\langle \boldsymbol{\alpha}, T_n \widetilde{\mathsf{h}}_k \rangle_{\ell^2(G)}\, T_n\mathsf{g}_k(m,h)\,.
\]
In the perfect reconstruction setting, for any $\boldsymbol{\alpha}\in \ell^2(G)$ we have
\begin{equation}
\label{perfect}
\boldsymbol{\alpha}=\sum_{k=1}^K\sum_{n\in N}\langle \boldsymbol{\alpha}, T_n \widetilde{\mathsf{h}}_k \rangle_{\ell^2(G)}\, T_n\mathsf{g}_k\quad \text{in $\ell^2(G)$}\,.
\end{equation}
Given  $K$ sequences $\mathsf{f}_k\in \ell^2(G)$, $k=1,2,\dots,K$, our main tasks now are: $(i)$ to characterize the  sequence $\big\{T_n \mathsf{f}_k\big\}_{n\in N;\,k=1,2,\dots K}$ as a frame for $\ell^2(G)$, and $(ii)$ to find its dual frames having the form $\big\{T_n \mathsf{g}_k\big\}_{n\in N;\,k=1,2,\dots K}$. 

\medskip

To the first end we consider a $K$-channel  analysis filter bank with analysis filters $\mathsf{h}_k:=\widetilde{\mathsf{f}}_k$, $k=1,2,\dots,K$; let $\mathbf{H}(\gamma)$ be its associated $K\times L$ polyphase matrix \eqref{Hpol}. First, we check that \eqref{Hpol} is:
\begin{equation}
\label{Hbf}
\mathbf{H}(\gamma)=\Big(\,\overline{\widehat{\mathsf{f}}_{k,h_i}(\gamma)}\,\Big)_{\substack{k=1,2,\ldots,K \\ i=1,2,\ldots, L}}\,.
\end{equation} 
Indeed, for $k=1,2,\dots,K$ and $i=1,2,\ldots, L$ we have
\[
\begin{split}
\widehat{\mathsf{h}}_{k,h_{i}}(\gamma)&=\sum_{n\in N} \mathsf{h}_{k,h_{i}}(n) 
\langle -n,\gamma \rangle=\sum_{n\in N} \mathsf{h}_{k}[(-n,h_{i})^{-1}] 
\langle -n,\gamma \rangle=\sum_{n\in N} \tilde{\mathsf{f}}_{k}[(-n,h_{i})^{-1}] 
\langle -n,\gamma \rangle\\
&=\sum_{n\in N} \overline{\mathsf{f}_{k}(-n,h_{i}) }\langle -n,\gamma \rangle= \overline{\sum_{n\in N} \mathsf{f}_{k}(n,h_{i}) 
\langle -n,\gamma \rangle} = \overline{\widehat{\mathsf{f}}_{k,h_{i}}(\gamma)}\,,\quad \gamma\in \widehat{N}\,.
\end{split}
\]
Next, we consider its associated constants 
\[
A_{\mathbf{H}}:=\einf_{\gamma \in \widehat{N}} \lambda_{\min}\big[\mathbf{H}^*(\gamma)\mathbf{H}(\gamma)\big]\quad \text{and}\quad B_{\mathbf{H}}:=\esup_{\gamma \in \widehat{N}} \lambda_{\max}\big[\mathbf{H}^*(\gamma)\mathbf{H}(\gamma)\big]\,.
\]
\begin{teo}
\label{besselframe}
For $\mathsf{f}_k$ in $\ell^2(G)$, $k=1,2,\dots,K$, consider the associated matrix $\mathbf{H}(\gamma)$ given in \eqref{Hbf}. Then,
\begin{enumerate}
\item The sequence $\big\{T_n \mathsf{f}_k\big\}_{n\in N;\,k=1,2,\dots K}$ is a Bessel sequence for $\ell^2(G)$ if and only if $B_{\mathbf{H}}<\infty$.
\item The sequence $\big\{T_n \mathsf{f}_k\big\}_{n\in N;\,k=1,2,\dots K}$ is a frame for $\ell^2(G)$ if and only if the inequalities $0<A_{\mathbf{H}}\le B_{\mathbf{H}}<\infty$ hold.
\end{enumerate}
\end{teo}
\begin{proof}
Using Plancherel theorem \cite[Theorem 4.25]{folland:95}, for each $\boldsymbol{\alpha}\in \ell^2(G)$ we get
\[
\begin{split}
\langle \boldsymbol{\alpha},T_{n}\mathsf{f}_{k}\rangle_{\ell^2(G)}&=
\sum_{h\in H} \langle \boldsymbol{\alpha}_{h},\mathsf{f}_{k,h}(\cdot-n)\rangle_{\ell^2(N)}=
\sum_{h\in H} \int_{\widehat{N}} \widehat{\boldsymbol{\alpha}}_{h}(\gamma)\overline{\widehat{\mathsf{f}}_{k,h}(\gamma)\langle -n,\gamma \rangle}d\gamma\\
&=\int_{\widehat{N}} \sum_{h\in H} \widehat{\boldsymbol{\alpha}}_{h}(\gamma)\overline{\widehat{\mathsf{f}}_{k,h}(\gamma)} \, \overline{ \langle -n,\gamma \rangle} d\gamma=  \int_{\widehat{N}} \mathbf{H}_k(\gamma) \mathbf{A}(\gamma)\overline{ \langle -n,\gamma \rangle} d\gamma\,,
\end{split}
\]
where $\mathbf{A}(\gamma)=\big(\widehat{\boldsymbol{\alpha}}_{h_1}(\gamma), \widehat{\boldsymbol{\alpha}}_{h_2}(\gamma),\dots , \widehat{\boldsymbol{\alpha}}_{h_L}(\gamma)\big)^\top$ and $\mathbf{H}_k(\gamma)$ denotes the $k$-th row of $\mathbf{H}(\gamma)$.

Since  $\big\{\langle -n,\gamma \rangle\big\}_{n\in N}$ is an orthonormal basis for $L^2(\widehat{N})$, in case that
$\mathbf{H}(\gamma)\mathbf{A}(\gamma)\in L_{K}^2(\widehat{N})$ we have
\[
\sum_{k=1}^K \sum_{n\in N} | \langle \boldsymbol{\alpha},T_{n}\mathsf{f}_{k} \rangle|^2=
\sum_{k=1}^K \int_{\widehat{N}} \big|\mathbf{H}_k(\gamma) \mathbf{A}(\gamma)\ \big|^2 d\gamma=
\int_{\widehat{N}} \big\| \mathbf{H}(\gamma) \mathbf{A}(\gamma) \big\|^2 d\gamma\,.
\]
If $B_{\mathbf{H}}<\infty$, having in mind that $\|\boldsymbol{\alpha}\|^2_{\ell^2(G)}=\|\mathbf{A}\|^2_{L_{L}^2(\widehat{N})}=
\int_{\widehat{N}} \big\| \mathbf{A}(\gamma) \big\|^2 d\gamma$, the above equality and the Rayleigh-Ritz theorem \cite[Theorem 4.2.2]{horn:99} prove that
$\{T_{n}\mathsf{f}_{k}\}_{n\in N;\,k=1,2,\ldots,K}$ is a  Bessel sequence for $\ell^2(G)$ with Bessel bound less or equal than $B_{\mathbf{H}}$. 

On the other hand, if $K<B_{\mathbf{H}}$ then there exists a set $\Omega\subset \widehat{N}$ having null measure such that $\lambda_{\max}\big[\mathbf{H}^*(\gamma)\mathbf{H}(\gamma)\big]>K$ for $\gamma\in \Omega$. Consider 
$\boldsymbol{\alpha}$ such that its associated $\mathbf{A}(\gamma)$ 
is $0$ if $\gamma\notin \Omega$, and $\mathbf{A}(\gamma)$ is a unitary eigenvector corresponding to the largest eigenvalue of $\mathbf{H}^*(\gamma)\mathbf{H}(\gamma)$ if $\gamma\in \Omega$. Thus we have that
\[
\sum_{k=1}^K \sum_{n\in N} | \langle \boldsymbol{\alpha},T_{n}\mathsf{f}_{k} \rangle|^2= \int_{\widehat{N}} 
\big\| \mathbf{H}(\gamma)
 \mathbf{A}(\gamma) \big\|^2 d\gamma> K \int_{\widehat{N}} 
\big\| 
 \mathbf{A}(\gamma) \big\|^2 d\gamma= K \, \|\boldsymbol{\alpha}\|^2_{\ell^2(G)}
\] 
As a consequence, if $B_{\mathbf{H}}=\infty$ the sequence is not Bessel, and if $B_{\mathbf{H}}<\infty$ the optimal bound is precisely $B_{\mathbf{H}}$. 

Similarly, by using inequality $\big\| \mathbf{H}(\gamma) \mathbf{A}(\gamma) \big\|^2\ge \lambda_{\min}\big[\mathbf{H}^*(\gamma)\mathbf{H}(\gamma)\big] \big\| \mathbf{A}(\gamma) \big\|^2$, and that equality holds whenever $\mathbf{A}(\gamma)$ is a unitary eigenvector corresponding to the smallest eigenvalue of 
$\mathbf{H}^*(\gamma)\mathbf{H}(\gamma)$ one proves the other inequality in part 2.
\end{proof}

\begin{cor}
\label{bessel}
The sequence $\big\{T_n \mathsf{f}_k\big\}_{n\in N;\,k=1,2,\dots K}$ is a Bessel sequence for $\ell^2(G)$ if and only if for each $k=1,2,\dots, K$ and $i=1,2,\dots,L$ the function $\widehat{\mathsf{f}}_{k,h_i}$ belongs to $L^\infty(\widehat{N})$.
\end{cor}
\begin{proof}
It is  a direct consequence of the equivalence between the spectral and Frobenius norms for matrices \cite{horn:99}.
\end{proof}

\medskip

To the second end, a $K$-channel filter bank formalism allows, in a similar manner, to obtain properties in $\ell^2(G)$ of the sequences $\big\{T_n \mathsf{f}_k\big\}_{n\in N;\,k=1,2,\dots K}$ and $\big\{T_n \mathsf{g}_k\big\}_{n\in N;\,k=1,2,\dots K}$.  In case they are Bessel sequences for $\ell^2(G)$, the idea is to consider a $K$-channel filter bank \eqref{fbank} where the analysis filters are $\mathsf{h}_k:=\widetilde{\mathsf{f}}_k$ and the synthesis filters are $ \mathsf{g}_k$, $k=1, 2,\dots,K$. As a consequence, the corresponding polyphase matrices $\mathbf{H}(\gamma)$ and  $\mathbf{G}(\gamma)$,  given in \eqref{Hpol} and \eqref{Gpol} are, 
\begin{equation}
\label{HGpol}
\mathbf{H}(\gamma)=\Big(\,\overline{\widehat{\mathsf{f}}_{k,h_i}(\gamma)}\,\Big)_{\substack{k=1,2,\ldots,K \\ i=1,2,\ldots, L}}\quad \text{and}\quad \mathbf{G}(\gamma)=\Big(\widehat{\mathsf{g}}_{h_i,k}(\gamma)\Big)_{\substack{i=1,2,\ldots,L \\ k=1,2,\ldots, K}}\,,\quad \gamma\in \widehat{N}\,.
\end{equation}
\begin{teo}
\label{translates}
Let $\big\{T_n \mathsf{f}_k\big\}_{n\in N;\,k=1,2,\dots K}$ and $\big\{T_n \mathsf{g}_k\big\}_{n\in N;\,k=1,2,\dots K}$ be two  Bessel sequences for $\ell^2(G)$, and $\mathbf{H}(\gamma)$ and $\mathbf{G}(\gamma)$ their associated matrices \eqref{HGpol}. Under the above circumstances we have:
\begin{enumerate}[(a)]
\item The sequences $\big\{T_n \mathsf{f}_k\big\}_{n\in N;\,k=1,2,\dots K}$ and $\big\{T_n \mathsf{g}_k\big\}_{n\in N;\,k=1,2,\dots K}$ are dual frames 
 for $\ell^2(G)$ if and only if condition $\mathbf{G}(\gamma)\mathbf{H}(\gamma)=\mathbf{I}_{L}$ a.e. $\gamma\in \widehat{N}$ holds.
\item The sequences $\big\{T_n \mathsf{f}_k\big\}_{n\in N;\,k=1,2,\dots K}$ and $\big\{T_n \mathsf{g}_k\big\}_{n\in N;\,k=1,2,\dots K}$ are biorthogonal sequences in 
$\ell^2(G)$ if and only if condition $\mathbf{H}(\gamma)\mathbf{G}(\gamma)=\mathbf{I}_{K}$ a.e. $\gamma\in \widehat{N}$ holds.
\item The sequences $\big\{T_n \mathsf{f}_k\big\}_{n\in N;\,k=1,2,\dots K}$ and $\big\{T_n \mathsf{g}_k\big\}_{n\in N;\,k=1,2,\dots K}$ are dual Riesz bases for 
$\ell^2(G)$ if and only if $K=L$ and $\mathbf{G}(\gamma)=\mathbf{H}(\gamma)^{-1}$ a.e. $\gamma\in \widehat{N}$.
\item The sequence $\big\{T_n \mathsf{f}_k\big\}_{n\in N;\,k=1,2,\dots K}$ is an $A$-tight frame for $\ell^2(G)$ if and only if condition $\mathbf{H}^*(\gamma)\mathbf{H}(\gamma)=A\mathbf{I}_{L}$ a.e. $\gamma\in \widehat{N}$ holds.
\item The sequence $\big\{T_n \mathsf{f}_k\big\}_{n\in N;\,k=1,2,\dots K}$ is an orthonormal basis for $\ell^2(G)$ if and only if $K=L$ and $\mathbf{H}^*(\gamma)=\mathbf{H}(\gamma)^{-1}$ a.e. $\gamma\in \widehat{N}$.
\end{enumerate}
\end{teo}
\begin{proof}
Having in mind \eqref{perfect} and Corollary \ref{bessel}, part $(a)$ is nothing but Theorem \ref{PR}. \\ 
The output of the analysis filter bank \eqref{fbank} corresponding to the input $\mathsf{g}_{k'}$  is a $K$-vector which $k$-entry is
\[
c_{k,k'}(m)=\downarrow_H\!(\mathsf{g}_{k'}*_N\mathsf{h}_k)(m)=\langle \mathsf{g}_{k'}, T_m \widetilde{\mathsf{h}}_k \rangle_{\ell^2(G)}=\langle \mathsf{g}_{k'}, T_m \mathsf{f}_k \rangle_{\ell^2(G)}\,,
\]
and whose $N$-Fourier transform is $\mathbf{C}_{k'}(\gamma)= \mathbf{H}(\gamma)\, \mathbf{G}_{k'}(\gamma)$ a.e. 
$\gamma \in \widehat{N}$,
where $\mathbf{G}_{k'}$ is the $k'$-column of the matrix $\mathbf{G}(\gamma)$.
Note that $\big\{T_n \mathsf{f}_k\big\}_{n\in N;\,k=1,2,\dots K}$ and $\big\{T_n \mathsf{g}_k\big\}_{n\in N;\,k=1,2,\dots K}$ are biorthogonal if and only if 
$\langle \mathsf{g}_{k'}, T_m \mathsf{f}_k \rangle_{\ell^2(G)}=\delta(k-k')\delta(m)$. Therefore, the sequences $\big\{T_n \mathsf{f}_k\big\}_{n\in N;\,k=1,2,\dots K}$ and $\big\{T_n \mathsf{g}_k\big\}_{n\in N;\,k=1,2,\dots K}$ are biorthogonal if and only if $\mathbf{H}(\gamma)\mathbf{G}(\gamma)=\mathbf{I}_{K}$. Thus, we have proved $(b)$. \\
 Having in mind \cite[Theorem 7.1.1]{ole:16}, from $(a)$ and $(b)$ we obtain $(c)$. \\
 We can read the frame operator corresponding to the sequence $\big\{T_n \mathsf{f}_k\big\}_{n\in N;\,k=1,2,\dots, K}$, i.e., 
\[
\mathcal{S}(\boldsymbol{\alpha})=\sum_{k=1}^K\sum_{n\in N}\langle \boldsymbol{\alpha}, T_n  \mathsf{f}_{k}\rangle_{\ell^2(G)}\, T_n\mathsf{f}_k,\quad \boldsymbol{\alpha}\in \ell^2(G)\,,
\]
as the output of the filter bank \eqref{fbank}, whenever $\mathsf{h}_{k}=\widetilde{\mathsf{f}}_k$ and $\mathsf{g}_{k}=\mathsf{f}_{k}$, for the input $\boldsymbol{\alpha}$. For this filter bank,  the $(k,h_l)$-entry of the analysis polyphase matrix $\mathbf{H}(\gamma)$ is $\overline{\widehat{\mathsf{f}}_{k,h_l}(\gamma)}$ and
the $(h_l,k)$-entry of the synthesis polyphase matrix $\mathbf{G}(\gamma)$ is $\widehat{\mathsf{f}}_{k,h_l}(\gamma)$; in other words, $\mathbf{G}(\gamma)=\mathbf{H}^*(\gamma)$. Hence, the sequence $\big\{T_n \mathsf{f}_k\big\}_{n\in N;\,k=1,2,\dots, K}$ is an $A$-tight  frame for $\ell^2(G)$, i.e.,
\[
\boldsymbol{\alpha}=\frac{1}{A}\sum_{k=1}^K\sum_{n\in N}\langle \boldsymbol{\alpha}, T_n  \mathsf{f}_{k}\rangle_{\ell^2(G)}\, T_n\mathsf{f}_k,\quad \boldsymbol{\alpha} \in \ell^2(G)\,,
\]
if and only if $\mathbf{H}^*(\gamma) \,\mathbf{H}(\gamma) =A\mathbf{I}_L$ for all $\gamma\in \widehat{N}$. Thus, we have proved $(d)$. \\
Finally, from $(c)$ and $(d)$ the sequence $\big\{T_n \mathsf{f}_k\big\}_{n\in N;\,k=1,2,\dots K}$ is an orthonormal system if and only if 
$\mathbf{H}^*(\gamma)=\mathbf{H}(\gamma)^{-1}$ a.e. $\gamma\in \widehat{N}$.
\end{proof}
%%%%%%%%%%%%%%%%%%%%%%%%%%%%%%%
\section{Getting on with sampling}
\label{section4}
%%%%%%%%%%%%%%%%%%%%%%%%%%%%%%%
Suppose that $\big\{U(n,h)\big\}_{(n,h)\in G}$ is a unitary representation of the group $G=N\rtimes_\phi H$ on a separable Hilbert space $\Hc$, and assume that for a fixed $a\in \Hc$ the sequence $\big\{U(n,h)a\big\}_{(n,h)\in G}$ is a Riesz sequence for $\Hc$ (see Ref.~\cite[Theorem A]{barbieri:15}). Thus, we consider the $U$-invariant subspace in $\Hc$
\[
\Ac_a=\Big\{ \sum_{(n,h)\in G} \alpha(n,h)\,U(n,h)a\,\, :\,\, \{\alpha(n,h)\}_{(n,h)\in G}\in \ell^2(G) \Big\}\,.
\]
For $K$ fixed elements $b_k\in \Hc$, $k=1,2, \dots ,K$, non necessarily in $\Ac$, we consider for each $x\in \Ac$ its generalized samples defined as
\begin{equation}
\label{samples}
\Lc_k x(m):=\big\langle x, U(m, 1_H)\,b_k \big\rangle_\Hc\,,\quad \text{ $m\in N$ and $k=1,2,\dots, K$}\,.
\end{equation}
The problem is the stable recovery of any $x\in \Ac$ from the data $\big\{ \Lc_k x(m)\big\}_{m\in N;\,k=1,2,\dots, K}$. 

\medskip

In what follows, we propose a solution involving a perfect reconstruction $K$-channel filter bank. First, we express the samples in a more suitable manner. Namely, for each $x=\sum_{(n,h)\in G} \alpha(n,h)\,U(n,h)\,a$ in $\Ac_a$ we have
\[
\begin{split}
\Lc_k x(m)&= \sum_{(n,h)\in G} \alpha(n,h) \big\langle U(n,h)\,a, U(m, 1_H)\,b_k \big\rangle \\
& = \sum_{(n,h)\in G} \alpha(n,h) \big\langle a, U\big[(n,h)^{-1}\cdot(m, 1_H)\big]\,b_k \big\rangle=\downarrow_H\!(\boldsymbol{\alpha}*\mathsf{h}_k)(m)\,, \quad m\in N\,,
\end{split}
\]
where $\boldsymbol{\alpha}=\{\alpha(n,h)\}_{(n,h)\in G} \in \ell^2(G)$, and  $\mathsf{h}_k(n,h):=\big\langle a, U(n,h)\,b_k\big\rangle_\Hc$ also belongs to 
$\ell^2(G)$ for each $k=1,2, \dots ,K$. 

Suppose also that there exists  a perfect reconstruction $K$-channel filter-bank with analysis filters the above $\mathsf{h}_k$ and synthesis filters $\mathsf{g}_k$, $k=1,2, \dots ,K$, such that the sequences $\big\{T_n \widetilde{\mathsf{h}}_k\big\}_{n\in N;\,k=1,2,\dots K}$ and $\big\{T_n \mathsf{g}_k\big\}_{n\in N;\,k=1,2,\dots K}$ are Bessel sequences for $\ell^2(G)$. Having in mind \eqref{perfect}, for each 
$\boldsymbol{\alpha}=\{\alpha(n,h)\}_{(n,h)\in G}$ in $\ell^2(G)$ we have
\begin{equation}
\label{sperfect}
\boldsymbol{\alpha}=\sum_{k=1}^K\sum_{n\in N}\downarrow_H\!(\boldsymbol{\alpha}*\mathsf{h}_k)(n)\, T_n\mathsf{g}_k 
=\sum_{k=1}^K\sum_{n\in N} \Lc_k x(n)\,T_n\mathsf{g}_k \quad \text{in $\ell^2(G)$}\,.
\end{equation}

In order to derive a sampling formula in $\Ac_a$, we consider the natural isomorphism $\Tc_{U,a}\,:\,\ell^2(G) \rightarrow \Ac_a$ which maps the usual orthonormal basis 
$\{\boldsymbol{\delta}_{(n,h)}\}_{(n,h)\in G}$ for $\ell^2(G)$ onto the Riesz basis $\big\{U(n,h)\,a\big\}_{(n,h)\in G}$ for $\Ac_a$, i.e., 
\[
\Tc_{U,a}\,:\, \boldsymbol{\delta}_{(n,h)} \longmapsto U(n,h) a\,\, \text{ for each $(n,h)\in G$}\,.
\]
This isomorphism $\Tc_{U,a}$ possesses the following shifting property:
\begin{lema}
For each $m\in N$, consider the translation operator $T_m$ operator defined in \eqref{translation}. For each  $m\in N$, the following shifting property holds
\begin{equation}
\label{shifting}
\Tc_{U,a}\big(T_m \mathsf{f}\big)=U(m, 1_H) \big(\Tc_{U,a}\mathsf{f}\big)\,, \quad \mathsf{f}\in \ell^2(G)\,.
\end{equation}
\end{lema}
\begin{proof}
For each $\boldsymbol{\delta}_{(n,h)}$ it is easy to check that $T_m\boldsymbol{\delta}_{(n,h)}=\boldsymbol{\delta}_{(m+n,h)}$. Hence,
\[
\Tc_{U,a}\big(T_m\boldsymbol{\delta}_{(n,h)}\big)=U(m+n,h)\,a=U(m,1_H)U(n,h)\,a=U(m, 1_H) \big(\Tc_{U,a}\boldsymbol{\delta}_{(n,h)}\big)\,.
\]
A continuity argument proves the result for all $\mathsf{f}$ in $\ell^2(G)$.
\end{proof}
Now for each $x=\Tc_{U,a} \boldsymbol{\alpha}\in \Ac_a$, applying the isomorphism $\Tc_{U,a}$ and the shifting property \eqref{shifting} in \eqref{sperfect}, we get for each $x\in \Ac_a$ the expansion
\begin{equation}
\label{sformula}
\begin{split}
x&=\sum_{k=1}^K\sum_{n\in N} \Lc_k x(n)\,\Tc_{U,a}\big(T_n\mathsf{g}_k\big)=\sum_{k=1}^K\sum_{n\in N} \Lc_k x(n)\,U(n,1_H)\big(\Tc_{U,a}\mathsf{g}_k\big)\\
&=\sum_{k=1}^K\sum_{n\in N} \Lc_k x(n)\,U(n,1_H)c_{k,\mathsf{g}} \quad \text{ in $\Hc$}\,,
\end{split}
\end{equation}
where $c_{k,\mathsf{g}}=\Tc_{U,a}\mathsf{g}_k$, $k=1,2,\dots, K$. In fact, the following sampling theorem in the subspace $\Ac_a$ holds:
\begin{teo}
\label{tsampling}
For $K$ fixed $b_k\in \Hc$, let $\Lc_k: N \rightarrow \CC$ be its associated $U$-system defined in \eqref{samples} with corresponding $\mathsf{h}_k \in \ell^2(G)$, $k=1,2,\dots,K$. Assume that its polyphase matrix $\mathbf{H}(\gamma)$ given in \eqref{Hpol} has all its entries in $L^\infty(\widehat{N})$. The following statements are equivalent:
\begin{enumerate}
\item The constant $A_{\mathbf{H}}=\displaystyle{\einf_{\gamma \in \widehat{N}} \lambda_{\min}\big[\mathbf{H}^*(\gamma)\mathbf{H}(\gamma)\big]>0}$.
\item There exist $\mathsf{g}_k$ in $\ell^2(G)$, $k=1,2,\dots,K$, such that the associated polyphase matrix $\mathbf{G}(\gamma)$ given in \eqref{Gpol} has all its entries in $L^\infty(\widehat{N})$, and it satisfies $\mathbf{G}(\gamma)\mathbf{H}(\gamma)=\mathbf{I}_{L}$ a.e. $\gamma\in \widehat{N}$.
\item There exist $K$ elements $c_k\in \Ac_a$ such that the sequence $\big\{U(n,1_H)c_k\big\}_{n\in N;\,k=1,2,\dots, K}$ is a frame for $\Ac_a$ and for each $x\in \Ac_a$ the sampling formula
\begin{equation}
\label{sformula2}
x=\sum_{k=1}^K\sum_{n\in N} \Lc_k x(n)\,U(n,1_H)c_k \quad \text{ in $\Hc$}
\end{equation}
holds.
\item There exists a frame $\big\{C_{k,n}\big\}_{n\in N;\,k=1,2,\dots, K}$ for $\Ac_a$ such that for each $x\in \Ac_a$ the expansion
\[
x=\sum_{k=1}^K\sum_{n\in N} \Lc_k x(n)\,C_{k,n} \quad \text{ in $\Hc$}
\]
holds.
\end{enumerate}
\end{teo}
\begin{proof}
$(1)$ implies $(2)$. The $L\times K$ Moore-Penrose pseudo-inverse $\mathbf{H}^\dag (\gamma)$ of $\mathbf{H}(\gamma)$ is given by $\mathbf{H}^\dag (\gamma)=\big[\mathbf{H}^*(\gamma)\,\mathbf{H}(\gamma)\big]^{-1}\,\mathbf{H}^*(\gamma)$. Its entries are essentially bounded in $\widehat{N}$ since the entries of 
$\mathbf{H}(\gamma)$ belong to $L^\infty(\widehat{N})$ and $\det^{-1}\big[\mathbf{H}^* (\gamma)\, \mathbf{H}(\gamma)\big]$ is essentially bounded $\widehat{N}$ since $0<A_{\mathbf{H}}$. Besides, $\mathbf{H}^\dag (\gamma)\mathbf{H}(\gamma)=\mathbf{I}_{L}$\, a.e. $\gamma\in \widehat{N}$. The inverse $N$-Fourier transform in $L^2(\widehat{N})$ of the $k$-th column of $\mathbf{H}^\dag (\gamma)$ gives $\mathsf{g}_k$, $k=1,2,\dots,K$. 

\medskip

\noindent $(2)$ implies $(3)$. According to Theorems \ref{besselframe} and \ref{translates} the sequences $\big\{T_n \widetilde{\mathsf{h}}_k\big\}_{n\in N;\,k=1,2,\dots K}$ and $\big\{T_n\mathsf{g}_k\big\}_{n\in N;\,k=1,2,\dots K}$ form a pair of dual frames for $\ell^2(G)$. 
We deduce the expansion sampling expansion as in \eqref{sformula}. Besides, the sequence $\big\{U(n,1_H)c_{k,\mathsf{g}}\big\}_{n\in N;\,k=1,2,\dots, K}$ is a frame for $\Ac_a$.

\medskip

\noindent Obviously, $(3)$ implies $(4)$. Finally, $(4)$ implies $(1)$. Applying $\Tc_{U,a}^{-1}$ we get that the sequences $\big\{T_n \widetilde{\mathsf{h}}_k\big\}_{n\in N;\,k=1,2,\dots K}$ and $\{ \Tc_{U,a}^{-1}(C_{k,n})\}_{n\in N;\,k=1,2,\dots, K}$ form a pair of dual frames for $\ell^2(G)$; in particular, by using Theorem \ref{besselframe} we obtain that $0<A_{\mathbf{H}}$.
\end{proof}

All the possible solutions of $\mathbf{G}(\gamma)\mathbf{H}(\gamma)=\mathbf{I}_{L}$ a.e. $\gamma\in \widehat{N}$ with entries in $L^\infty(\widehat{N})$ are given in terms of the Moore-Penrose pseudo inverse by the $L\times K$ matrices $\mathbf{G}(\gamma):=\mathbf{H}^\dag(\gamma)+\mathbf{U}(\gamma)\big[\mathbf{I}_K-\mathbf{H}(\gamma)\mathbf{H}^\dag(\gamma)\big]$,
where $\mathbf{U}(\gamma)$ denotes any $L\times K$ matrix with entries in $L^\infty(\widehat{N})$.

Notice that $K\geq L$ where $L$ is the order of the group $H$. In case $K=L$, we obtain:
\begin{cor}
In the case $K=L$, assume that its polyphase matrix $\mathbf{H}(\gamma)$ given in \eqref{Hpol} has all entries in $L^\infty(\widehat{N})$. The following statements are equivalents:
\begin{enumerate}
\item The constant $A_{\mathbf{H}}=\displaystyle{\einf_{\gamma \in \widehat{N}} \lambda_{\min}\big[\mathbf{H}^*(\gamma)\mathbf{H}(\gamma)\big]>0}$.
\item There exist $L$ unique elements $c_k$, $k=1,2,\dots, L$, in $\Ac_a$ such that the associated sequence
$\big\{U(n,1_H)c_{k}\big\}_{n\in N;\,k=1,2,\dots, L}$ is a Riesz basis for $\Ac_a$ and the sampling formula
\[
x=\sum_{k=1}^L\sum_{n\in N} \Lc_k x(n)\,U(n,1_H)c_{k}\ \quad \text{ in $\Hc$}
\]
holds for each $x\in \Ac_a$.
\end{enumerate}
Moreover, the interpolation property $\Lc_k c_{k'}(n)=\delta_{k,k'}\delta_{n,0_N}$, where $n\in N$ and $k,k'=1,2,\dots,L$, holds.
\end{cor}
\begin{proof}
In this case, the square matrix $\mathbf{H}(\gamma)$ is invertible and the result comes out from Theorem \ref{translates}. From the uniqueness of the coefficients in a Riesz basis we get the interpolation property.
\end{proof}
Denote $H=\{h_1,h_2,\dots,h_L\}$; for a fixed $b\in \Hc$,  we consider the samples 
\[
\Lc_k x(m):=\big\langle x, U(m, h_k)b\big\rangle\,,\quad \text{ $m\in N$ and $k=1,2,\dots, L$}\,, 
\]
of any $x\in \Ac_a$. Since $U(m, h_k)b=U(m,1_H)U(0_N,h_k)b=U(m,1_H)b_k$, where $b_k:=U(0_N,h_k)b$, $k=1,2,\dots, L$, we are in a particular case of \eqref{samples} with $K=L$.
%%%%%%%%%%%%%%%%%%%%%%%%%%%%%%
\subsection{An example involving crystallographic groups}
%%%%%%%%%%%%%%%%%%%%%%%%%%%%%%
The Euclidean motion group $E(d)$ is the semi-direct product $\RR^d \rtimes_{\phi} O(d)$ corresponding to the homomorphism $\phi : O(d) \rightarrow Aut(\RR^d)$ given by $\phi_{A}(x) = Ax$, where $A\in O(d)$ and $x\in \RR^d$. The composition law on $E(d) = \RR^d \rtimes_{\phi} O(d)$ reads $(x, A) \cdot  (x', A') = (x + Ax', AA')$.

Let $M$ be a non-singular $d\times d$ matrix and $\Gamma$ a finite subgroup of $O(d)$ of order $L$ such that  $A(M\ZZ^d)=M\ZZ^d$ for each $A\in \Gamma$. We consider the {\em crystallographic group} $\Cc_{M,\Gamma}:=M\ZZ^d \rtimes_\phi \Gamma$ and its quasi regular representation (see Ref.~\cite{barbieri:15}) on $L^2(\RR^d)$
\[
U(n,A)f(t)=f[A^{\top}(t-n)]\,,\quad \text{$n\in M\ZZ^d$, $A\in \Gamma$ and $f\in L^2(\RR^d)$}\,.
\]
For a fixed $\varphi \in L^2(\RR^d)$ such that the sequence $\big\{ U(n,A)\varphi \big\}_{(n,A)\in \Cc_{M,\Gamma}}$ is a Riesz sequence for $L^2(\RR^d)$ we consider the $U$-invariant subspace in $L^2(\RR^d)$
\[
\begin{split}
\Ac_\varphi&=\Big\{\sum_{(n,A)\in \Cc_{M,\Gamma}} \alpha(n,A)\, \varphi [A^{\top}(t-n)] \,\,:\,\, \{\alpha(n,A)\}\in \ell^2(\Cc_{M,\Gamma})\Big\}\\
&=\Big\{\sum_{(n,A)\in \Cc_{M,\Gamma}} \alpha(n,A)\, \varphi (At-n) \,\,:\,\, \{\alpha(n,A)\}\in \ell^2(\Cc_{M,\Gamma})\Big\}\,.
\end{split}
\]
Choosing $K$ functions $b_k\in  L^2(\RR^d)$, $k=1,2,\dots,K$, we consider the average samples of $f\in \Ac_\varphi$
\[
\Lc_k f(n)=\langle f, U(n, I) b_k\rangle=\langle f, b_k(\cdot-n)\rangle\,,\quad n\in M\ZZ^d\,.
\]
Under the hypotheses in Theorem \ref{tsampling}, there exist $K\geq L$ sampling functions $\psi_k \in \Ac_\varphi$ for $k=1,2,\dots,K$, such that the sequence $\{\psi_k(\cdot-n)\}_{n\in M\ZZ^d;\,k=1,2,\dots,K}$ is a frame for $\Ac_\varphi$, and the sampling expansion
\begin{equation}
\label{scristal}
f(t)=\sum_{k=1}^K \sum_{n\in M\ZZ^d} \big\langle f, b_k(\cdot-n)\big\rangle_{L^2(\RR^d)} \, \psi_k(t-n) \quad \text{in $L^2(\RR^d)$}
\end{equation}
holds. 

If the generator $\varphi \in C(\RR^d)$ and the function $t\mapsto \sum_n|\varphi(t-n)|^2$ is bounded on $\RR^d$, a standard argument shows that $\Ac_\varphi$ is a reproducing kernel Hilbert space (RKHS) of continuous functions in $L^2(\RR^d)$. As a consequence, convergence in $L^2(\RR^d)$-norm implies pointwise convergence which is absolute and uniform on 
$\RR^d$.

\medskip

Notice that the infinite dihedral group $D_\infty=\ZZ\rtimes_\phi \ZZ_2$ is a particular crystallographic group with lattice 
$\ZZ$ and $\Gamma=\ZZ_2$. Its quasi regular representation on $L^2(\RR)$ reads
\[
U(n,0)f(t)=f(t-n)\quad \text{and}\quad U(n,1)f(t)=f(-t+n)\,,\quad n\in \ZZ \text{ and } f\in L^2(\RR)\,.
\]
So we could obtain  sampling formulas as \eqref{scristal} for $K\geq 2$ average functions $b_k$.

\medskip

The quasi regular unitary representation of a crystallographic group $\Cc_{M,\Gamma}$ on $L^2(\RR^d)$ motivates  the next section:
%%%%%%%%%%%%%%%%%%%%%%%%
\subsection{The case of pointwise samples}
%%%%%%%%%%%%%%%%%%%%%%%%
Let $\{U(n,h)\}_{(n,h)\in G}$ be a unitary representation of the group $G=N\rtimes_\phi H$ on the Hilbert space $\Hc=L^2(\RR^d)$. If the generator $\varphi\in L^2(\RR^d)$ satisfies that, for each $(n,h)\in G$, the function $U(n,h)\varphi$ is continuous on $\RR^d$, and the condition
\[
\sup_{t\in \RR^d} \sum_{(n,h)\in G} \big|[U(n,h)\varphi](t)\big|^2<\infty\,,
\]
then the subspace $\Ac_\varphi$ is a RKHS of continuous functions in $L^2(\RR^d)$; proceeding as in \cite{zhou:99} one can prove that the above conditions are also necessary.

For $K$ fixed points $t_k\in \RR^d$, $k=1,2,\dots,K$, we consider for each $f\in \Ac_\varphi$ the new samples given by
\begin{equation}
\label{samples2}
\Lc_kf(n):=\big[U(-n,1_H)f\big](t_k)\,,\quad n\in N\, \text{ and }\, k=1,2,\cdots,K\,.
\end{equation}
For each $f=\sum_{(m,h)\in G} \alpha(m,h)\,U(m,h)\,\varphi$ in $\Ac_\varphi$ and $k=1,2,\dots,K$ we have
\[
\begin{split}
\Lc_kf(n)&=\Big[\sum_{(m,h)\in G} \alpha(m,h)\,U[(-n,1_H)\cdot(m,h)]\,\varphi\Big](t_k)\\
&=\sum_{(m,h)\in G} \alpha(m,h)\big[U(m-n,h)\varphi\big](t_k)=\big\langle \boldsymbol{\alpha}, T_n\mathsf{h}_k\big\rangle_{\ell^2(G)}\,,\quad n\in N\,,
\end{split}
\]
where $\boldsymbol{\alpha}=\{\alpha(m,h)\}_{(m,h)\in G}$ and $\mathsf{h}_k(m,h):=\overline{\big[U(m,h)\varphi\big](t_k)}$, $(m,h)\in G$. Notice that $\mathsf{h}_k$ belongs to $\ell^2(G)$, $k=1,2,\cdots,K$. As a consequence, under the hypotheses in Theorem \ref{tsampling} (on these new $\mathsf{h}_k\in \ell^2(G)$, $k=1,2,\dots,K$) a sampling formula as \eqref{sformula2} holds for the  data sequence $\big\{ \Lc_k f(n)\big\}_{n\in N;\,k=1,2,\dots, K}$ defined in \eqref{samples2}.

\medskip

In the particular case of the quasi regular representation of a crystallographic group $\Cc_{M,\Gamma}=M\ZZ^d\rtimes_\phi \Gamma$, for each $f\in \Ac_\varphi$ the samples \eqref{samples2} read
\[
\Lc_kf(n)=\big[U(-n,I)f\big](t_k)=f(t_k+n)\,, \quad n\in M\ZZ^d \,\text{ and }\, k=1,2,\dots, K\,.
\]
Thus (under hypotheses in Theorem \ref{tsampling}), there exist $K$ functions $\psi_k\in \Ac_\varphi$, $k=1,2,\dots, K$, such that for each $f\in \Ac_\varphi$ the sampling formula
\[
f(t)=\sum_{k=1}^K \sum_{n\in M\ZZ^d} f(t_k+n)\,\psi_k(t-n)\,,\quad t\in \RR^d
\]
holds. The convergence of the series in the $L^2(\RR^d)$-norm sense implies pointwise convergence which is absolute and uniform on $\RR^d$.

\medskip

\noindent{\bf Acknowledgments:}
This work has been supported by the grant MTM2017-84098-P from the Spanish {\em Ministerio de Econom\'{\i}a y Competitividad (MINECO)}.

%%%%%%%%%%%%%%%%%%%%%%%%%%%%%%%%%%%%%%%%%%%%%%%%%%%%%%%%%%%%%%%%%%%%%

%%%%%%%%%%%%%%%%%%%%%%%%%
\end{document}